\documentclass[a4paper,10pt]{amsart}
\usepackage[utf8]{inputenc}
\usepackage{amsfonts}
\usepackage{amsmath}
\usepackage{amssymb}
\usepackage{graphicx}
\usepackage[all]{xy}
\usepackage{hyperref}

\theoremstyle{plain}
\newtheorem{lemma}{Lemma}[section]
\newtheorem{Theorem}[lemma]{Theorem}
\newtheorem{theorem}[lemma]{Theorem}

\newtheorem{corollary}[lemma]{Corollary}
\newtheorem{Corollary}[lemma]{Corollary}
\newtheorem{proposition}[lemma]{Proposition}
\newtheorem{Proposition}[lemma]{Proposition}

\theoremstyle{definition}
\newtheorem{definition}[lemma]{Definition}

\newtheorem*{Acknowledgements}{Acknowledgements}
\newtheorem{remark}[lemma]{Remark}

\DeclareFontFamily{OT1}{pzc}{}
\DeclareFontShape{OT1}{pzc}{m}{it}{<-> s * pzcmi7t}{}
\DeclareMathAlphabet{\mathpzc}{OT1}{pzc}{m}{it}
\DeclareMathOperator{\VCyc}{\mathpzc{VCyc}}
\DeclareMathOperator{\Fin}{\mathpzc{Fin}}
\DeclareMathOperator{\Triv}{\mathpzc{Triv}}

\DeclareMathOperator{\uE}{\underline{E}}
\DeclareMathOperator{\E}{E}
\DeclareMathOperator{\EF}{E_\mathcal{F}}

\DeclareMathOperator{\cd}{cd}
\DeclareMathOperator{\ucd}{\underline{cd}}
\DeclareMathOperator{\uucd}{\underline{\underline{cd}}}
\DeclareMathOperator{\gd}{gd}
\DeclareMathOperator{\ugd}{\underline{gd}}
\DeclareMathOperator{\uugd}{\underline{\underline{gd}}}
\DeclareMathOperator{\gdF}{gd_\mathcal{F}}

\DeclareMathOperator{\cdF}{cd_\mathcal{F}}
\DeclareMathOperator{\pdF}{pd_\mathcal{F}}

\DeclareMathOperator{\id}{id}

\newcommand{\OFG}{\mathop{\mathcal{O}_\mathcal{F}G}}

\DeclareMathOperator{\pt}{pt}

\DeclareMathOperator{\FP}{FP}
\DeclareMathOperator{\uFP}{\underline{FP}}

\newcommand*{\longhookrightarrow}{\ensuremath{\lhook\joinrel\relbar\joinrel\rightarrow}}

\newcommand*{\longtwoheadrightarrow}{\ensuremath{\relbar\joinrel\twoheadrightarrow}}

\newcommand*{\RR}{\ensuremath{\mathbb{R}}}

\newcommand*{\ZZ}{\ensuremath{\mathbb{Z}}}
\newcommand*{\NN}{\ensuremath{\mathbb{N}}}

\newcommand*{\QQ}{\ensuremath{\mathbb{Q}}}

\DeclareMathOperator{\ZZF}{\mathbb{Z}_\mathcal{F}}

\DeclareMathOperator{\Sym}{Sym}
\DeclareMathOperator{\vmv}{\vert \mathcal{M} \vert}

\begin{document}

\author{Simon St. John-Green}
\address[Simon St. John-Green]{School of Mathematics, University of Southampton, University Road, Southampton SO17 1BJ, UK \emph{(S.StJohn-Green@soton.ac.uk)}}
\title{Centralisers in Houghton's groups}
 
{\abstract{We calculate the centralisers of elements, finite subgroups and virtually cyclic subgroups of Houghton's group $H_n$.  We discuss various Bredon (co)homological finiteness conditions satisfied by $H_n$ including the Bredon (co)homological dimension and $\uFP_n$ conditions, which are analogs of the ordinary cohomological dimension and $\FP_n$ conditions respectively.
}}

\maketitle

Houghton's group $H_n$ was introduced in \cite{Houghton-FirstCohomologyOfAGroupWithPermModCoeff}, as an example of a group acting on a set $S$ with $H^1(H_n, A \otimes \ZZ[S]) = A^{n-1}$ for any Abelian group $A$.   

In \cite{Brown-FinitenessPropertiesOfGroups}, Brown used an important new technique to show that the Thompson-Higman groups $F_{n,r}, T_{n,r}$ and $V_{n,r}$ were $\FP_\infty$.  In the same paper he showed that Houghton's group $H_n$ is interesting from the viewpoint of cohomological finiteness conditions, namely $H_n$ is $\FP_{n-1}$ but not $\FP_n$.  Thompsons group $F$ was previously shown by different methods to be $\FP_\infty$ \cite{BrownGeoghegan-AnInfiniteDimensionTFFPinftyGroup}, thus providing the first known example of a torsion-free $\FP_\infty$ group with infinite cohomological dimension.

There has been recent interest in the structure of the centralisers of Thompson's groups, in \cite{MartinezNucinkis-GeneralizedThompsonGroups} the centralisers of finite subgroups of generalisations of Thompson's groups $T$ and $V$ are calculated and this data is used to give information about Bredon (co)homological finiteness conditions satisfied by these groups.  The results obtained in \cite[Theorem 4.4, 4.8]{MartinezNucinkis-GeneralizedThompsonGroups} have some similarity with those obtained here.  In \cite{BleakEtAl-CentralizersInVn}, a description of centralisers of elements in the Thompson-Higman group $V_n$ is given.  

In Section \ref{section:review of bredon finiteness} we give the necessary background on Bredon (co)homological finiteness conditions.  Section \ref{section:centralisersOfFiniteSubgroups} contains an analysis of the centralisers of finite subgroups in Houghton's group.  As Corollary \ref{cor:stabOfH is FPn-1 not FPn} we obtain that centralisers of finite subgroups are $\FP_{n-1}$ but not $\FP_n$.  This should be compared with \cite{KochloukovaDessislavaMartinezPerez-CentralisersOfFiniteSubgroupsInSolubleGroups} where examples are given of soluble groups of type $\FP_n$ with centralisers of finite subgroups that are not $\FP_n$, and also with \cite{MartinezPerezNucinkis-VirtuallySolubleGroupsOfTypeFPinfty}, where it is shown that in virtually soluble groups of type $\FP_\infty$ the centralisers of all finite subgroups are of type $\FP_\infty$.

In Section \ref{section:centralisersOfElements} our analysis is extended to arbitrary elements and virtually cyclic subgroups.  Using this information elements in $H_n$ are constructed whose centralisers are $\FP_i$  for any $0 \le i \le n-3$.  In Section \ref{section:Browns model} the space that Brown constructed in \cite{Brown-FinitenessPropertiesOfGroups}, in order to prove that $H_n$ is $\FP_{n-1}$ but not $\FP_n$, is shown to be a model for $\uE H_n$, the classifying space for proper actions of $H_n$.  Finally Section \ref{section:finitenessConditionsSatisfied} contains a discussion of Bredon (co)homological finiteness conditions are satisfied by Houghton's group.  
Namely we show in Proposition \ref{prop:HisnotQuasiFP0} that $H_n$ is not quasi-$\uFP_0$ and in Proposition \ref{prop:cd=gd=n} that the Bredon cohomological dimension and the Bredon geometric dimension with respect to the family of finite subgroups are both equal to $n$.  See Section \ref{section:review of bredon finiteness} for the definitions of quasi-$\uFP_n$ and of Bredon cohomological and geometric dimension.

Fixing a natural number $n >1$, define \emph{Houghton's group} $H_n$ to be the group of permutations of $S = \mathbb{N} \times \{ 1, \ldots, n\}$ which are ``eventually translations'', ie. for any given permutation $h \in H_n$ there are collections $\{z_1, \ldots, z_n\} \in \NN^n$ and $\{m_1, \ldots, m_n\} \in \ZZ^n$ with 
\begin{equation}\label{eventuallyATranslation}
h(i,x) = (i+m_x, x) \text{ for all }x \in \{1, \ldots, n \} \text{ and all }i \ge z_x
\end{equation}

Define a map $\phi$ as follows:
\begin{align}\label{def:phi}
 \phi : H_n &\to \{ (m_1, \ldots, m_n) \in \mathbb{Z}^n \: : \: \sum m_i = 0\} \cong \ZZ^{n-1} \\
 \phi : h &\mapsto (m_1, \ldots, m_n)
\end{align}
Its kernel is exactly the permutations which are ``eventually zero'' on $S$, ie. the infinite symmetric group $\Sym_\infty$ (the finite support permutations of a countable set).  

\begin{Acknowledgements}
 The author would like to thank his supervisor Brita Nucinkis for her encouragement and for many enlightening conversations and Conchita Martinez-Perez for her very helpful comments.  The author would also like to thank the referee for making extremely detailed and useful comments, including suggesting the graph $\Gamma$ used in Section \ref{section:centralisersOfElements}.
\end{Acknowledgements} 

\section{A review of Bredon (co)homological finiteness conditions}
\label{section:review of bredon finiteness}

Throughout this section $G$ is a discrete group and $\mathcal{F}$ is a family of subgroups of $G$ which is closed under taking subgroups and conjugation.  The \emph{orbit category}, denoted $\OFG$, is the small category whose objects are the transitive $G$-sets $G/H$ for $H \in \mathcal{F}$ and whose arrows are all $G$-maps between them.  Any $G$-map $G/H \to G/K$ is determined entirely by the image of the coset $H$ in $G/K$, and $H \mapsto xK$ determines a $G$-map if and only if $ x^{-1}Hx \le K$.

An $\OFG$-module, or \emph{Bredon module}, is a contravariant functor from $\OFG$ to the category of Abelian groups.  As such, the category {\bf $\OFG$-Mod} of $\OFG$ modules is Abelian and exactness is defined pointwise - a short exact sequence 
$$ M^\prime \longrightarrow M \longrightarrow M^{\prime\prime} $$
is exact if and only if 
$$ M^\prime(G/H) \longrightarrow M(G/H) \longrightarrow M^{\prime\prime}(G/H) $$
is exact for all $H \in \mathcal{F}$.  The category of $\OFG$-modules can be shown to have enough projectives.  If $\Omega_1$ and $\Omega_2$ are $G$-sets then we denote by $\ZZ[\Omega_1, \Omega_2]$ the free Abelian group on the set of all $G$-maps $\Omega_1 \to \Omega_2$. If $H \in \mathcal{F}$, the $\OFG$ module $\ZZ[-,G/K]$ defined by 
$$\ZZ[-,G/K](G/H) = \ZZ[G/H,G/K]$$
is free and such modules serve as the building blocks of free $\OFG$-modules.  More precisely, any free $\OFG$-module is a direct sum of such modules.  

An $\OFG$-module $M$ is said to be finitely generated if it admits an epimorphism from some free $\OFG$-module 
$$\bigoplus_{i \in I} \ZZ[-,G/H_i] \longtwoheadrightarrow M$$
with $I$ a finite set.

We denote by $\ZZF$ the $\OFG$-module taking all objects to $\ZZ$ and all arrows to the identity map.  Analogously to ordinary group cohomology we define the Bredon \emph{cohomological dimension} of a $\OFG$-module $M$ to be the shortest length of a projective resolution of $M$ by $\OFG$-modules, and the cohomological dimension of a group $G$ to be the shortest length of a projective resolution of the $\OFG$-module $\ZZF$.  These two integers are denoted $\pdF M$ and $\cdF G$, if $\mathcal{F}=\Fin$ (the family of finite subgroups) then the notation $\ucd G$ is used and if $\mathcal{F} = \VCyc$ (the family of virtually cyclic subgroups) then the notation $\uucd G$ is used.

The Bredon \emph{geometric dimension} of a group $G$, denoted $\gdF G$, is defined to be the minimal dimension of a model for $\EF G$.  Recall that a $G$-CW-complex is a CW-complex with a cellular, rigid $G$-action, where a rigid action is one where the pointwise and setwise stabilisers of all cells coincide.  A model for $\EF G$ is defined to be a $G$-CW-complex $X$ such that 

 $$X^H \simeq \left\{ \begin{array}{l l} \pt & \text{ if } H \in \mathcal{F} \\ \emptyset & \text{ if } H \notin \mathcal{F} \end{array}\right.$$

By an application of the equivariant Whitehead theorem \cite[Theorem 2.4]{Lueck} this is unique up to $G$-homotopy equivalence.  In the case where $\mathcal{F} = \Triv$, the family consisting of only the trivial subgroup, a model for $\E_{\Triv}G$ is the universal cover $\E G$ of an Eilenberg-Mac Lane space $\text{K}(G,1)$.  An $n$-dimensional model for $\EF G$ gives rise to a free resolution of $\OFG$-modules $C_*$ by setting $C_n(G/H) = K_n(X^H)$, where $K_n$ denotes the cellular chain complex of a CW-complex.  Immediately we deduce that $\cdF G \le \gdF G$.  

A theorem of L\"uck and Meintrup gives an inequality in the other direction:
\begin{Theorem}\cite[Theorem 0.1]{LuckMeintrup-UniversalSpaceGrpActionsCompactIsotropy}  $\gd_\mathcal{F} G \le \max \{ \cd_\mathcal{F} G, 3 \}$
\end{Theorem}
If $\mathcal{F} = \Fin$ we denote the geometric dimension by $\ugd G$, and if $\mathcal{F} = \VCyc$, by $\uugd G$.  Dunwoody has shown that $\ucd G = 1$ implies that $\ugd G = 1$ \cite{Dunwoody-AccessabilityAndGroupsOfCohomologicalDimensionOne}, hence $\ucd G = \ugd G$ unless $\ucd G = 2$ and $\ugd G = 3$.  Brady, Leary and Nucinkis show in \cite{BradyLearyNucinkis-AlgAndGeoDimGroupsWithTorsion} that this can indeed happen.

There are many groups for which good models for $\uE G$ are known, \cite{Luck-SurveyOnClassifyingSpaces} is a good reference.  

The $\uFP_n$-conditions are natural generalisations of the $\FP_n$ conditions of ordinary group cohomology.  An $\OFG$-module $M$ is $\uFP_n$ if it admits a projective resolution by $\mathcal{O}_{\Fin} G$-modules, which is finitely generated in all dimensions $\le n$.  A group $G$ is $\uFP_n$, if $\ZZ_{\Fin}$ is $\uFP_n$.  

The following lemma details an alternative algebraic description of the condition $\uFP_n$ which is easier to calculate.

\begin{Proposition}
\cite[Lemma 3.1, Lemma 3.2]{KMN-CohomologicalFinitenessConditionsForElementaryAmenable}
\begin{enumerate}
 \item $G$ is $\uFP_0$ if and only if $G$ has finitely many conjugacy classes of finite subgroups.
 \item An $\OFG$-module $M$ is $\uFP_n$ ($n \ge 1$) if and only if $G$ is $\uFP_0$ and $M(G/K)$ is of type $\FP_n$ over the Weyl group $WK = N_GK/ K$ for all finite subgroups $K \le G$.
\end{enumerate}
\end{Proposition}

\begin{Corollary}\label{cor:uFP_n equivalent conditions}
The following are equivalent for a group $G$:
\begin{enumerate}
 \item $G$ is $\uFP_n$.
 \item $G$ is $\uFP_0$ and the Weyl groups $WK = N_GK/K$ are $\FP_n$ for all finite subgroups $K$.
 \item $G$ is $\uFP_0$ and the centralisers $C_G K$ are $\FP_n$ for all finite subgroups $K$.
\end{enumerate}
\end{Corollary}
\begin{proof}
 By the previous proposition (1) and (2) are equivalent.  To see the equivalence of (2) and (3) consider the short exact sequence
$$0 \longrightarrow K \longrightarrow N_G K  \longrightarrow WK \longrightarrow 0 $$
$K$ is finite and hence $\FP_\infty$, so $WK$ is $\FP_n$ if and only if $N_G K$ is $\FP_n$ \cite[Proposition 2.7]{Bieri-HomDimOfDiscreteGroups}.  $K$ is finite, so $C_G K$ is finite index in $N_G K $ \cite[1.6.13]{Robinson} and as such $C_G K$ is $\FP_n$ if and only if $N_G K$ is $\FP_n$.  Hence $WK$ is $\FP_n$ if and only if $C_G K$ is $\FP_n$.
\end{proof}

The condition that a group $G$ has only finitely many conjugacy classes of finite subgroups is extremely strong, in \cite{MartinezNucinkis-GeneralizedThompsonGroups} the weaker condition quasi-$\uFP_n$ is introduced.  

\begin{definition}
\begin{enumerate}
 \item $G$ is quasi-$\uFP_0$  if and only if there are finitely many conjugacy classes of finite subgroups isomorphic to a given finite subgroup.
 \item $G$ is quasi-$\uFP_n$ if and only if $G$ is quasi-$\uFP_0$ and $WK$ is $\FP_n$ for every finite $K \le G$.
\end{enumerate}
\end{definition}

Many results about finiteness in ordinary group cohomology carry over into the Bredon case, for example in in \cite[Section 5]{MartinezNucinkis-GeneralizedThompsonGroups}, versions of the Bieri-Eckmann criterion for both $\uFP_n$ and quasi-$\uFP_n$ $\OFG$-modules are shown to hold (see \cite[Section 1.3]{Bieri-HomDimOfDiscreteGroups} for the classical case).

\section{Centralisers of finite subgroups in \texorpdfstring{$H_n$}{Hn}}\label{section:centralisersOfFiniteSubgroups}
First we recall some properties of group actions on sets, before specialising to Houghton's group.
\begin{proposition}\label{prop:groupActionsOnSets}
 If $G$ is a group acting on a countable set $X$ and $H$ is any subgroup of $G$ then
\begin{enumerate}
 \item If $x$ and $y$ are in the same $G$-orbit then their isotropy subgroups $G_x$ and $G_y$ are $G$-conjugate.
 \item If $g \in C_G(H)$ then $H_{gx} = H_x$ for all $x \in X$.
 \item Partition $X$ into $\{X_a\}_{a=1}^t$, where $t \in \NN \cup \{\infty\}$, via the equivalence relation $x \sim y$ if and only if $H_x$ is $H$-conjugate to $H_y$.  Any two points in the same $H$-orbit will lie in the same partition and any $c \in C_G(H)$ maps $X_a$ onto $X_a$ for all $a$.
 \item Let $G$ act faithfully on $X$, with the property that for all $g \in G$ and $X_a \subseteq X$ as in the previous section, there exists a group element $g_a \in G$ which fixes $X \setminus X_a$ and acts as $g$ does on $X_a$.  Then $C_G(H) = C_1 \times \cdots \times C_t$ where $C_a$ is the subgroup of $C_G(H)$ acting trivially on $X \setminus X_a$.
\end{enumerate}
\end{proposition}
\begin{proof}(1) and (2) are standard results.
\begin{enumerate}
 \setcounter{enumi}{2} 
\item This follows immediately from (1) and (2).
 \item This follows from (3) and our new assumption on $G$:  Let $c \in C_G(H)$ and $c_a$ be the element given by the assumption.  Since the action of $G$ on $X$ is faithful, $c_a$ is necessarily unique.  That the action is faithful also implies $c = c_1 \cdots c_t$ and that any two $c_a$ and $c_b$ commute in $G$ because they act non-trivially only on distinct $X_a$.  Thus we have the necessary isomorphism $C_G(H) \longrightarrow C_1 \times \cdots \times C_t$.
\end{enumerate}
\end{proof}
Let $Q \le H_n$ be a finite subgroup of Houghton's group $H_n$ and $S_Q  = S \setminus S^Q$ the set of points of $S$ which are \emph{not fixed} by $Q$.  $Q$ being finite implies $\phi(Q) = 0$ as any element $q$ with $\phi(q) \neq 0$ necessarily has infinite order.  For every $q \in Q$ there exists $\{z_1, \ldots, z_n \} \in \NN^n$ such that
$$
q(i,x) = (i, x) \text{ if } i \ge z_x
$$
Taking $z_x^\prime$ to be the maximum of these $z_x$ over all elements in $Q$, then $Q$ must fix the set $\{(i,x)\: : \: i \ge z^\prime_x\}$ and in particular $S_Q \subseteq \{(i,x)\: : \: i < z^\prime_x\}$ is finite.

We need to see that the subgroup $Q \le H_n$ acting on the set $S$ satisfies the conditions of Proposition \ref{prop:groupActionsOnSets}(4).  We give the following lemma in more generality than is needed here, as it will come in useful later on.  That the action is faithful is automatic as an element $h \in H_n$ is uniquely determined by its action on the set $S$.
\begin{lemma}\label{lemma:prop conditions are satisfied}
 Let $Q \le H_n$ be a subgroup, which is either finite or of the form $F \rtimes \ZZ$ for $F$ a finite subgroup of $H_n$.  Partition $S$ with respect to $Q$  into sets $\{S_a\}_{a = 1}^t$ as in Proposition \ref{prop:groupActionsOnSets}(3) applied to the action of $H_n$ on $S$ and the subgroup $Q$ of $H_n$.  Then the conditions of Proposition \ref{prop:groupActionsOnSets}(4) are satisfied.
\end{lemma}
\begin{proof}
Fix $a \in \{1, \ldots, t\}$ and let $h_a$ denote the permutation of $S$ which fixes $S \setminus S_a$ and acts as $h$ does on $S_a$.  We wish to show that $h_a$ is an element of $H_n$.

There are only finitely many elements in $Q$ with finite order so as in the argument just before this lemma we may choose integers $z_x$ for $x \in \{1, \ldots, n\}$ such that if $q$ is a finite order element of $Q$ then $q(i,x) = (i,x)$ whenever $i \ge z_x$.  If $Q$ is a finite group then either:
\begin{itemize}
 \item $S_a$ is fixed by $Q$, in which case 
$$\{(i,x) \: : \: i \ge z_x \: , \: x \in \{1, \ldots, n \} \} \subseteq S_a$$
 so $h_a(i,x) = h(i,x)$ for all $i \ge z_x$.  In particular for large enough $i$, $h_a$ acts as a translation on $(i,x)$ and is hence an element of $H_n$.
\end{itemize}
Or
\begin{itemize}
 \item $S_a$ is not fixed by $Q$, in which case 
 $$S_a \subseteq \{(i,x) \: : \: i < z_x \: , \: x \in \{1, \ldots, n \} \} $$
In particular $S_a$ is finite and $h_a(i,x) = (i,x)$ for all $i \ge z_x$.  Hence $h_a$ is an element of $H_n$.
\end{itemize}

It remains to treat the case where $Q= F \rtimes \ZZ$.  Write $w$ for a generator of $\ZZ$ in $F \rtimes \ZZ$.  By choosing a larger $z_x$ if needed we may assume $w$ acts either trivially or as a translation on $(i,x)$ whenever $i \ge z_x$.  Hence for any $x \in \{1, \ldots, n\}$, the isotropy group in $Q$ of $\{(i,x) \: :\: i \ge z_x \}$ is either $F$ or $Q$.  

If $S_a$ has isotropy group $Q$ or $F$ then for some $x \in \{1, \ldots, n\}$, either:
\begin{itemize}
 \item $$S_a \cap \{(i,x) : i \ge z_x\} = \{(i,x) : i \ge z_x\}$$
 In which case $h_a(i,x) = h(i,x)$ for $i \ge z_x$.  In particular for large enough $i$, $h_a$ acts as a translation on $(i,x)$ and hence is an element of $H_n$.
\end{itemize}
Or
\begin{itemize} 
\item $$S_a \cap \{(i,x) : i \ge z_x\} = \emptyset$$
 In which case $h_a(i,x) = (i,x)$ for $i \ge z_x$.  In particular for large enough $i$, $h_a$ fixes $(i,x)$ and hence is an element of $H_n$.
\end{itemize}
If $S_a$ is the set corresponding to an isotropy group not equal to $F$ or $Q$ then 
$$S_a \subseteq \{ (i,x) \: : \: i \ge z_x \: , \: x \in \{1, \ldots, n\} \}$$
So $h_a$ fixes $(i,x)$ for $i \ge z_x$ and hence $h_a$ is an element of $H_n$.
\end{proof}

Partition $S$ into disjoint sets according to the $Q$-conjugacy classes of the stabilisers, as in Proposition \ref{prop:groupActionsOnSets}(3).  The set with isotropy in $Q$ equal to $Q$ is $S^Q$ and since $S_Q$ is finite the partition is finite, thus
$$ S = S^Q \cup S_1 \cup \cdots \cup S_t$$
Proposition \ref{prop:groupActionsOnSets}(4) gives that  
$$C_{H_n}(Q) = H_n|_{S^Q} \times C_1 \times \ldots \times C_t $$ 
where each $C_a$ acts only on $S_a$ and leaves $S^Q$ and $S_b$ fixed for $a \neq b$ (where $a, b \in \{1, \ldots, t\}$).  The first element of the direct product decomposition is the subgroup of $C_{H_n}(Q)$ acting only on $S^Q$ and leaving $S \setminus S^Q$ fixed.  This is $H_n|_{S^Q}$ ($H_n$ restricted to $S^Q$) because, as the action of $Q$ on $S^Q$ is trivial, any permutation of $S^Q$ will centralise $Q$.  Choose a bijection $S^Q \to S$ such that for all $x$, $(i,x) \mapsto (i + m_x, x)$ for large enough $i$ and some $m_x \in \ZZ$, this induces an isomorphism between $H_n|_{S^Q}$ and $H_n$.

To give an explicit definition of the group $C_a$ we need three lemmas.

\begin{lemma}\label{lemma:Ca iso T}
 $C_a$ is isomorphic to the group $T$ of $Q$-set automorphisms of $S_a$.
\end{lemma}
\begin{proof}
 An element $c \in C_a$ determines a $Q$-set automorphism of $S_a$, giving a map $C_a \to T$.  Since the action of $C_a$ on $S_a$ is faithful this map is injective.  Any $Q$-set automorphism $\alpha$ of $S_a$ may be extended to a $Q$-set automorphism of $S$, where $\alpha$ acts trivially on $S \setminus S_a$.  Since $S_a$ is a finite set, $\alpha$ acts trivially on $(i,x)$ for large enough $i$ and any $x \in \{1, \ldots, n\}$, and hence $\alpha$ is an element of $H_n$.  Finally, since $\alpha$ is a $Q$-set automorphism $q \alpha s = \alpha q s$, equivalently $\alpha^{-1} q \alpha s = s$, for all $s \in S$ and $q \in Q$, showing that $\alpha \in C_a$ and so the map $C_a \to T$ is surjective.
\end{proof}

\begin{lemma}\label{lemma:S_a is disjoint union of Q/Qa}
 $S_a$ is $Q$-set isomorphic to the disjoint union of $r$ copies of $Q / Q_a$, where $Q_a$ is an isotropy group of $S_a$ and $r = \vert S_a \vert / \vert Q : Q_a \vert$.
\end{lemma}
\begin{proof}
 $S_a$ is finite and so splits as a disjoint union of finitely many $Q$-orbits.  Choose orbit representatives $\{s_1, \ldots, s_r \} \subset S_a$ for these orbits, these $s_k$ may be chosen to have the same $Q$-stabilisers: If $Q_{s_1} \neq Q_{s_2}$ then there is some $q \in Q$ such that $Q_{qs_2} = qQ_{s_2} q^{-1} = Q_{s_1}$ (the partitions $S_a$ were chosen to have this property by Proposition \ref{prop:groupActionsOnSets}), iterating this procedure we get a set of representatives who all have isotropy group $Q_{s_1}$.  Now set $Q_a = Q_{s_1}$ and note that there are $\vert Q:Q_a \vert $ elements in each of the $Q$-orbits so $ r \vert Q : Q_a \vert = \vert S_a \vert$.
\end{proof}

Recall that if $G$ is any group and $r \ge 1$ is some natural number then the wreath product $ G \wr \Sym_r $ is the semi-direct product
$$ G \wr \Sym_r = \prod_{k = 1}^r G \rtimes \Sym_r $$
where the symmetric group $\Sym_r$ acts by permuting the factors in the direct product.  

Recall also that for any subgroup $H$ of a group $G$, the Weyl group $W_G H$ is defined to be $W_G H = N_G H / H$.
\begin{lemma}\label{lemma:description of Ca}
The group $C_a$ is isomorphic to the wreath product $ W_QQ_a \wr \Sym_r $, where $Q_a$ is some isotropy group of $S_a$ and $r = \vert S_a \vert / \vert Q : Q_a \vert$.
\end{lemma}
\begin{proof}
 Using Lemmas \ref{lemma:Ca iso T} and \ref{lemma:S_a is disjoint union of Q/Qa}, $C_a$ is isomorphic to the group of $Q$-set automorphisms of the disjoint union of $r$ copies of $Q/Q_a$.

 To begin, we show the group of automorphisms of the $Q$-set $Q / Q_a$ is isomorphic to $W_QQ_a$.  An automorphism $\alpha : Q/Q_a \to Q/Q_a$ is determined by the image $\alpha(Q_a) = qQ_a$ of the identity coset and such an element determines an automorphism if and only if $q^{-1} Q_a q \le Q_a $, equivalently $q \in N_Q Q_a$.  Since two elements $q_1,q_2 \in Q$ will determine the same automorphism if and only if $q_1 Q_a = q_2Q_a$, the group of $Q$-set automorphisms of $Q/Q_a$ is the Weyl group $W_QQ_a$.

 For the general case, note that if $c \in C_a$ then $c$ permutes the $Q$-orbits $\{Qs_1, \ldots, Qs_r\}$, so there is a map $C_a \to \Sym_r$.  Assume that the representatives $\{s_1, \ldots, s_r\}$ have been chosen, as in the proof of Lemma \ref{lemma:S_a is disjoint union of Q/Qa}, to have the same $Q$-stabilisers.  The map $\pi$ is split by the map 
\begin{align*}
 \iota:\Sym_r &\to C_a \\
 \sigma &\mapsto \left( \iota(\sigma):  qs_k \mapsto qs_{\sigma(k)} \text{ for all }q \in Q \right)
\end{align*}
Each $\iota(\sigma)$ is a well defined element of $H_n$ since 
$$qs_k = \tilde{q} s_k \Leftrightarrow \tilde{q}^{-1}q \in Q_{s_k} = Q_{s_{\sigma(k)}} \Leftrightarrow qs_{\sigma(k)} = \tilde{q} s_{\sigma(k)}$$

The kernel of the map $\pi$ is exactly the elements of $C_a$ which fix each $Q$-orbit but may permute the elements inside the $Q$-orbits, by the previous part this is exactly $\prod_{k = 1}^r W_QQ_a$.  For any $\sigma \in \Sym_r$, the element $\iota(\sigma)$ acts on $\prod_{k = 1}^r W_QQ_a$ by permuting the factors, so the group $C_a$ is indeed isomorphic to the wreath product.
\end{proof}

\begin{figure}[ht]
 \centering
 \includegraphics[width = 0.4\textwidth]{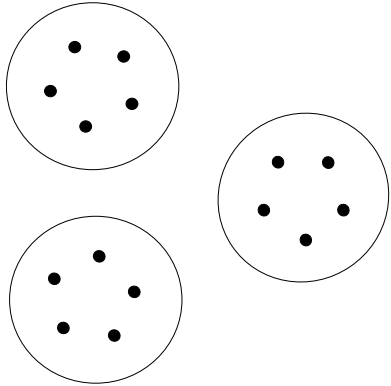}
 \caption{A representation of $S_a$.  The large circles are the sets $\{Qs_1 \ldots, Qs_r\}$ (in this figure $r = 3$).  Elements of $\Sym_r$ permute only the large circles, while elements of $\prod_{k = 1}^r W_QQ_a$ leave the large circles fixed and permute only elements inside them.}
\end{figure}

The centraliser $C_{H_n}Q$ can now be completely described.

\begin{proposition}\label{prop:stabOfH}
 The centraliser $C_{H_n}(Q)$ of any finite subgroup $Q \le H_n$ splits as a direct product 
$$C_{H_n}(Q) \cong  H_n|_{S^Q} \times C_1 \times \cdots \times C_t$$
where $H_n|_{S^Q} \cong H_n$ is Houghton's group restricted to $S^Q$ and for all $a \in \{1, \ldots, t\}$,
$$C_a \cong W_Q Q_a \wr \Sym_r $$
for $Q_a$ is an isotropy group of $S_a$ and r = $\vert S_a \vert / \vert Q : Q_a \vert$.  In particular $H_n$ is finite index in $C_{H_n}(Q)$.
\end{proposition}

\begin{proof}
We have already proven that 
$$C_{H_n}(Q) \cong  H_n|_{S^Q} \times C_1 \times \cdots \times C_t$$
and Lemma \ref{lemma:description of Ca} gives the required description of $C_a$.
\end{proof}

\begin{corollary}\label{cor:stabOfH is FPn-1 not FPn}
 If $Q$ is a finite subgroup of $H_n$ then the centraliser $C_{H_n}(Q)$ is $\FP_{n-1}$ but not $\FP_n$.
\end{corollary}
\begin{proof}
 $H_n$ is finite index in the centraliser $C_{H_n}(Q)$ by Proposition \ref{prop:stabOfH}.  Appealing to Brown's result \cite[5.1]{Brown-FinitenessPropertiesOfGroups} that $H_n$ is $\FP_{n-1}$ but not $\FP_n$, and that a group is $\FP_n$ if and only if a finite index subgroup is $\FP_n$ \cite[VIII.5.5.1]{Brown} we can deduce $C_{H_n}(Q)$ is $\FP_{n-1}$ but not $\FP_n$. 
\end{proof}

\section{Centralisers of Elements in \texorpdfstring{$H_n$}{Hn}}\label{section:centralisersOfElements}

If $q \in H_n$ is an element of finite order then the subgroup $Q = \langle q \rangle$ is a finite subgroup and the previous section may be used to describe the centraliser $C_{H_n}(q) = C_{H_n}(Q)$.  Thus for an element $q$ of finite order $C_{H_n}(q) \cong C \times H_n$ for some finite group $C$.

If $q \in H_n$ is an element of infinite order and $Q = \langle q \rangle$ then we may apply Proposition \ref{prop:groupActionsOnSets}(3) to split up $S$ into a disjoint collection $\{ S_a \: : \: a \in A \subseteq \NN\} \cup S^Q$ ($S^Q$ is the element of the collection associated to the isotropy group $Q$).  Assume that $S_0$ is the set associated to the trivial isotropy group.  Since $q$ is a translation on $(i,x) \in S = \NN \times \{1, \ldots, n\}$ for large enough $i$ and points acted on by such a translation have trivial isotropy, there are only finitely many elements of $S$ whose isotropy group is neither the trivial group nor $Q$.  Hence $S_a$ is finite for $a \neq 0$ and the set $A$ is finite.  From now on let $A = \{0, \ldots, t\}$.  We now use Lemma \ref{lemma:prop conditions are satisfied} and Proposition \ref{prop:groupActionsOnSets}(4) as in the previous section: $C_{H_n}(Q)$ splits as 
$$C_{H_n}(Q) \cong C_0 \times C_1 \times \cdots \times C_t \times H_n|_{S^Q}$$
Where $C_a$ acts only on $S_a$ and $H_n|_{S^Q}$ is Houghton's group restricted to $S^Q$.  Unlike in the last section, $H_n|_{S^Q}$ may not be isomorphic to $H_n$.  Let $J \subseteq \{1, \ldots, n\}$ satisfy
$$ x \in J \text{ if and only if } (i,x) \in S^Q \text{ for all }i \ge z_x \text{, some }z_x \in \NN $$
If $x \notin J$ then for large enough $i$, $q$ must act as a non-trivial translation on $(i,x)$, and the set $\left(\NN \times \{x\}  \right) \cap S^Q $ is finite.  Clearly $\vert J \vert \le n-2$, but different elements $q$ may give values $0 \le \vert J \vert \le n-2$.  In the case $\vert J \vert = 0$, $S^Q$ is necessarily finite and so $H_n|_{S^Q}$ is isomorphic to a finite symmetric group on $S^Q$.  It is also possible that $S^Q  = \emptyset$, in which case $H_n|_{S^Q}$ is just the trivial group.  If $\vert J \vert \neq 0$ then the argument proceeds as in the previous section by choosing a bijection
$$S^Q \to \NN \times J$$
such that $(i,x) \mapsto (i + m_x, x)$ for some $m_x \in \ZZ$ whenever $i$ is large enough and $x \in J$.  This set map induces a group isomorphism between $H_n|_{S^Q}$ and $H_{\vert J \vert}$ (Houghton's group on the set $J \times \NN$).   

Lemma \ref{lemma:description of Ca} describes the groups $C_a$ for $a \neq 0$, so it remains only to treat the case $a=0$.  We cannot use the arguments used for $a \neq 0$ here as the set $S_0$ is not finite, in particular Lemma \ref{lemma:Ca iso T} doesn't apply: Every $Q$-set isomorphism of $S_0$ is realised by an element of the infinite support permutation group on $S_0$, but there are $Q$-set isomorphisms of $S_0$ which are not realised by an element of $H_n$.

The next three lemmas are needed to describe $C_0$, this description will use the graph $\Gamma$ which we now describe.  The vertices of $\Gamma$ are those $x \in \{1, \ldots, n\}$ for which $q$ acts non-trivially on infinitely many elements of $ \NN \times \{ x\} $.  Equivalently, the vertices are the elements of $\{1, \ldots, n\} \setminus J$.  There is an edge from $x$ to $y$ in $\Gamma$ if there exists $s \in S_0$ and $N \in \NN$ such that for all $m \ge N$ we have $q^{-m}s \in \NN \times \{ x \} $ and $q^m s \in \NN \times \{ y\}$.  Let $\pi_0 \Gamma$ denote the path components of $\Gamma$, and for any vertex $x$ of $\Gamma$ denote by $[x]$ the element of $\pi_0 \Gamma$ corresponding to that vertex.

Let $z \in \NN$ be some integer such that for all $i \ge z$, $q$ acts trivially or as a translation on $(i, x)$ for all $x \in \{1, \ldots, n\}$.  Fix $z$ for the remainder of this section.

For each path component $[x]$ in $\pi_0 \Gamma$, let $S_0^{[x]}$ denote the smallest $Q$-subset of $S_0$ containing the set $\{(i,y) \: : \: i \ge z , \: y \in [x] \}$.  Note that $(i,y) \notin S_0^{[x]}$ for any $y \notin [x]$ and $i \ge z$, since if $(i,x)$ and $(j,y)$ are two elements of $S_0$ in the same $Q$-orbit with $i \ge z$ and $j \ge z$ then there is an edge between $x$ and $y$ in $\Gamma$:  If $(i,x) = q^k(j,y)$ and $q$ acts as a positive translation on the element $(i,x)$ then let $N=k$ and $s = (i,x)$, similarly for when $q$ acts as a negative translation.  This gives a $Q$-set decomposition of $S_0$ as:
$$S_0 = \coprod_{[x] \in \pi_0 \Gamma} S_0^{[x]} $$
Where $\coprod$ denotes disjoint union.

\begin{lemma}\label{lemma:decomposition of elements in S0x}
Let $[x] \in \pi_0 \Gamma$, if $C_0^{[x]}$ denotes the subgroup of $C_0$ which acts non-trivially only on $S_0^{[x]}$ then there is an isomorphism
$$C_0 \cong C_0^{[x_1]} \times \cdots \times C_0^{[x_r]} $$
Where $[x_1], [x_2], \ldots, [x_r]$ are all elements of $\pi_0 \Gamma$.
\end{lemma}
\begin{proof}
If $c \in C_0$ and $[x] \in \pi_0 \Gamma$ then let $c_{[x]}$ denote the permutation of $S$ such that $c_{[x]}$ acts as $c$ does on $S_0^{[x]}$, and acts trivially on $S \setminus S_0^{[x]}$.  We will show that $c_{[x]}$ is an element of $C_0$.  Since the action of $C_0$ on $S_0$ is faithful it follows that the elements $c_{[x]}$ and $c_{[y]}$ commute and 
$$c = c_{[x_1]} c_{[x_2]} \cdots c_{[x_r]} $$
which suffices to prove the lemma.  

Let $y \in \{1,\ldots,n\}$.  The element $c_{[x]}$ acts trivially on $(i,y)$ for $i \ge z$ if $y \notin [x]$ and acts as as $c$ does on $(i,y)$ for $i \ge z$ if $y \in [x]$, thus $c_{[x]}$ is an element of $H_n$.  Since $c_{[x]}$ is also a $Q$-set automorphism of $S$, $c_{[x]}$ is a member of $C_0$.
\end{proof}

\begin{lemma}\label{lemma:uniquness of action on S0x}
 Let $[x] \in \pi_0\Gamma$, $c \in C_0$, and $z^\prime \in \NN$ such that $c$ acts either trivially or as a translation on $(i,x)$ for all $x \in \{1, \ldots, n\}$ and $i \ge z^\prime$.  Then the action of $c$ on some element $(i, x) \in S$ for $i \ge z^\prime$ completely determines the action of $c$ on $S_0^{[x]}$.
\end{lemma}
\begin{proof}
Firstly, note that knowing the action of $c$ on some element $(i,x)$ for $i \ge z^\prime$ determines the action of $c$ on the set $\{(i,x) \: : \: i \ge z^\prime \}$, since we chose $z^\prime$ in order to have this property.

Let $y \in [x]$ such that there is an edge from $x$ to $y$, so there is a natural number $N$ and element $s \in S_0^{[x]}$ such that $ q^N s = (i,x)$ and $q^{-N} s = (j,y)$ for some natural numbers $i$ and $j$.  By choosing $N$ larger if necessary we can take $i, j \ge z^\prime$.  The action of $c$ on $(j,y)$ is now completely determined by the action on $(i,x)$, since
$$ c(j,y) = c q^{-2N}(i,x) = q^{-2N} c(i,x) $$
For any $y \in [x]$ there is a path from $x$ to $y$ in $\Gamma$, so we've determined the action of $c$ on the set $X = \{(j,y) \: : \: j \ge z^\prime \: , \: y \in [x]\}$.  If $s \in S_0^{[x]} \setminus X$ then, since $S_0^{[x]} \setminus X$ is finite, there is some integer $m$ with $q^m s = x \in X$.  So $c s = c q^{-m} x = q^{-m} c x $, which completely determines the action of $c$ on $s$.
\end{proof}

\begin{lemma}\label{lemma:C0x iso Z}
 For any $[x] \in \pi_0\Gamma$, there is an isomorphism
$$C_0^{[x]} \cong \ZZ $$
\end{lemma}
\begin{proof}
 By Lemma \ref{lemma:uniquness of action on S0x} the action is completely determined by the action on some element $(i,x)$ for large enough $i$, and the action on this element is necessarily by translation by some element $m_x(c)$.  This defines an injective homomorphism $C_0^{[x]} \to \ZZ$, sending $c \mapsto m_x(c)$.  Let $q_{[x]}$ be the element of $C_0^{[x]}$ described in the proof of Lemma \ref{lemma:decomposition of elements in S0x}, $q_{[x]}$ is a non-trivial element of $C_0^{[x]}$ so $C_0^{[x]}$ is mapped isomorphically onto a non-trivial subgroup of $\ZZ$.
\end{proof}

Combining Lemmas \ref{lemma:decomposition of elements in S0x} and \ref{lemma:C0x iso Z} shows $C_0 \cong \ZZ^r$ where $r = \vert \pi_0 \Gamma \vert$.  

Recall that the vertices of $\Gamma$ are indexed by the set $\{1, \ldots, n\} \setminus J$.  Since there are no isolated vertices in $\Gamma$, $\vert \pi_0 \Gamma \vert \le \lfloor (n- \vert J \vert ) / 2 \rfloor$ (where $\lfloor - \rfloor$ denotes the integer floor function).  Recalling that $0 \le \vert J \vert \le n-2$, the set $\{1, \ldots, n\} \setminus J$ is necessarily non-empty so $1 \le \vert \pi_0 \Gamma \vert $, combining these gives
$$1 \le \vert \pi_0 \Gamma \vert \le \lfloor (n-\vert J \vert)/2 \rfloor $$

We can now completely describe the centraliser $C_{H_n}(q)$.

\begin{theorem}\label{thm:stabOfInfiniteQ}
\begin{enumerate}
 \item If $q \in H_n$ is an element of finite order then 
 $$C_{H_n}(q) \cong H_n\vert_{S^Q} \times C_1 \times \cdots \times C_t$$
 Where $H_n\vert_{S^Q} \cong H_n$ is Houghton's group restricted to $S^Q$ and for all $a \in \{1, \ldots, t\}$, 
 $$ C_a \cong W_QQ_a \wr \Sym_r $$
 for $Q_a$ an isotropy group of $S_a$ and $r = \vert S_a \vert / \vert Q : Q_a \vert$.  In particular $H_n$ is finite index in $C_{H_n}Q$.
 \item If $q \in H_n$ is an element of infinite order then either
$$C_{H_n}(q) \cong H_k \times \ZZ^r \times C_1 \times \cdots \times C_t$$ 
 Or
$$C_{H_n}(q) \cong F \times \ZZ^r \times C_1 \times \cdots \times C_t$$ 
 Where $F$ is some finite symmetric group, $H_k$ is Houghton's group with $0 \le k \le n-2$, and the groups $C_a$ are as in the previous part.   In the first case $1 \le r \le \lfloor (n-k) / 2 \rfloor$, and in the second case $1 \le r \le \lfloor n / 2 \rfloor$.  
\end{enumerate}
\end{theorem}

 In Corollary \ref{cor:stabOfH is FPn-1 not FPn} it was proved that for an element $q$ of finite order, $C_{H_n}(q)$ is $\FP_{n-1}$ but not $\FP_n$.  The situation is much worse for elements $q$ of infinite order, in which case the centraliser may not even be finitely generated, for example when $n$ is odd and $q$ is the element acting on $S = \NN \times \{1, \ldots, n\}$ as
$$
q :\left\{ \begin{array}{l l}
  (i,x) \mapsto (i+1,x) & \text{ if }x \le (n-1)/2 \\
  (i,x) \mapsto (i-1,x) & \text{ if }(n+1)/2 \le x \le n-1 \text{ and } i \neq 0\\
  (0,x) \mapsto (0,x - ((n-1)/2)) & \text{ if } (n+1)/2 \le x \le n-1 \\
  (i,n) \mapsto (i,n)
   \end{array} \right.
$$
then the only fixed points are on the ray $\NN \times \{n \}$.  The argument leading up to Theorem \ref{thm:stabOfInfiniteQ} shows that the centraliser is a direct product of groups, one of which is Houghton's group $H_1$ which is isomorphic to the infinite symmetric group and hence not finitely generated.  In particular for this $q$, the centraliser $C_{H_n}(q)$ is not even $\FP_1$.  A similar example can easily be constructed when $n$ is even.

All the groups in the direct product decomposition from Theorem \ref{thm:stabOfInfiniteQ} except $H_k$ are $\FP_\infty$, being built by extensions from finite groups and free Abelian groups.  By choosing various infinite order elements $q$, for example by modifying the example of the previous paragraph, the centralisers can be chosen to be $\FP_k$ for $0 \le k \le n-3$.  The upper bound of $n-3$ arises because any infinite order element $q$ must necessarily be ``eventually a translation'' (in the sense of \eqref{eventuallyATranslation}) on  $\NN \times \{x \} $ for \emph{at least} two $x$.  As such the copy of Houghton's group in the centraliser can act on at most $n-2$ rays and is thus at largest $H_{n-2}$, which is $\FP_{n-3}$.  

\begin{corollary}\label{cor:stabOfVCycSubgroups}
 If $Q$ is an infinite virtually cyclic subgroup of $H_n$ then either
$$ C_{H_n}(Q) \cong H_k \times \ZZ^r \times C_1 \times \cdots \times C_t $$
Or 
$$ C_{H_n}(Q) \cong F \times \ZZ^r \times C_1 \times \cdots \times C_t $$
where the elements in the decomposition are all as in Theorem \ref{thm:stabOfInfiniteQ}.
\end{corollary}

This corollary can be proved by reducing to the case of Theorem \ref{thm:stabOfInfiniteQ}, but before that we require the following lemma.
\begin{lemma}
 Every infinite virtually cyclic subgroup $Q$ of $H_n$ is finite-by-$\ZZ$.
\end{lemma}
\begin{proof}
By \cite[Proposition 4]{JuanPinedaLeary-ClassifyingSpacesForVCsubgrps}, $Q$ is either finite-by-$\ZZ$ or finite-by-$\text{D}_\infty$ where $\text{D}_\infty$ denotes the infinite dihedral group, we show the latter cannot occur.  Assume that there is a short exact sequence 
$$0 \longrightarrow F \longhookrightarrow Q \stackrel{\pi}{\longrightarrow} \text{D}_\infty \longrightarrow 0 $$
regarding $F$ as a subgroup of $Q$.  Let $a,b$ generate $\text{D}_\infty$, so that 
$${\text{D}_\infty = \langle a,b \mid a^2 = b^2 = 1\rangle}$$
Let $p,q \in Q$ be lifts of $a, b$, such that $\pi(p) = a$, $\pi(q) = b$, then $p^2 \in F$.  Since $F$ is finite, $p^2$ has finite order and hence $p$ has finite order.  The same argument shows that $q$ has finite order.  $pq \in Q$ necessarily has infinite order as $\pi(pq)$ is infinite order in $D_\infty$.

However, since $p$ and $q$ are finite order elements of $H_n$, by the argument at the beginning of Section \ref{section:centralisersOfFiniteSubgroups} they both permute only a finite subset of $S$.  Thus $pq$ permutes a finite subset of $S$ and is of finite order, but this contradicts the previous paragraph.

\end{proof}
\begin{proof}[Proof of Corollary \ref{cor:stabOfVCycSubgroups}]
Using the previous lemma, write $Q$ as $Q = F \rtimes \ZZ$ where $F$ is a finite group.  As $F$ is finite, the set $S_F$ of points not fixed by $F$ is finite (see the argument at the beginning of Section \ref{section:centralisersOfFiniteSubgroups}).  Let $z \in \NN$ be such that for $i \ge z$, $F$ acts trivially on $(i,x)$ for all $x$, and $\ZZ$ acts on $(i,x)$ either trivially or as a translation.  Applying Lemma \ref{lemma:prop conditions are satisfied} and Proposition \ref{prop:groupActionsOnSets}, $S$ splits as a disjoint union 
$$S = S^Q \cup S_0 \cup S_1 \cup \cdots \cup S_t$$ 
where $S^Q$ is the fixed point set, $S_0$ is the set with isotropy group $F$ and the $S_a$ for $1 \le a \le t$ are subsets of $\{(i,x)\: :  \: i \le z \}$, and hence all finite.  By Proposition \ref{prop:groupActionsOnSets}, $C_{H_n}(Q)$ splits as a direct product
$$C = H_n|_{S^Q} \times C_0 \times C_1 \times \ldots \times C_t $$
where $H_n|_{S^Q}$ denotes Houghton's group restricted to $S^Q$.  The argument of Theorem \ref{thm:stabOfInfiniteQ} showing that $H_n|_{S^Q}$ is isomorphic to either a finite symmetric group or to $H_k$ for some $0 \le k \le n-2$ goes through with no change, as does the proof of the structure of the groups $C_a$ for $1 \le a \le t$.  It remains to observe that because every element in $S_0$ is fixed by $F$, any element of $H_n$ centralising $\ZZ$ and fixing $S \setminus S_0$ necessarily also centralises $Q$ and is thus a member of $C_0$.  This reduces us again to the case of Theorem \ref{thm:stabOfInfiniteQ} showing that $C_0 \cong \ZZ^r$ for some natural number $1 \le r \le \lfloor (n-k)/2 \rfloor$, or $1 \le r \le \lfloor n /2 \rfloor$ if $H_n\vert_{S^Q}$ is a finite symmetric group.
\end{proof}

\section{Brown's Model for \texorpdfstring{\underline{E}$H_n$}{the Classifying Space for Proper Actions of Hn}}\label{section:Browns model}

The main result of this section will be Corollary \ref{cor:Browns model is an EunderbarH}, where the construction of Brown \cite{Brown-FinitenessPropertiesOfGroups} used to prove that $H_n$ is $\FP_{n-1}$ but not $\FP_n$ is shown to be a model for $\uE H_n$.

\emph{Since the main objects of study in this section are monoids, maps are written from left to right.}

Write $\mathcal{M}$ for the monoid of injective maps $S \to S$ with the property that every permutation is ``eventually a translation'' (in the sense of \eqref{eventuallyATranslation}), and write $T$ for the free monoid generated by $\{t_1, \ldots, t_n\}$ where 
$$(i,x)t_y = \left\{ \begin{array}{c c}(i+1,x) & \text{ if } x = y \\ (i,x) & \text{ if } x \neq y \end{array}\right.$$
The elements of $T$ will be called \emph{translations}.  The map $\phi : H_n \to \ZZ^n$, defined in \eqref{def:phi}, extends naturally to a map $\phi : \mathcal{M} \to \ZZ^n$.  
Give  $\mathcal{M}$ a poset structure by setting $\alpha \le \beta$ if $\beta = t\alpha$ for some $t \in T$.  The monoid $\mathcal{M}$ can be given the obvious action on the right by $H_n$, which in turn gives an action of $H_n$ on the poset $(\mathcal{M}, \le)$ since $\beta = t\alpha$ implies $\beta h = t\alpha h$ for all $h \in H_n$.  Let $\vert \mathcal{M} \vert$ be the geometric realisation of this poset, namely simplicies in $\vert \mathcal{M} \vert$ are finite ordered collections of elements in $\mathcal{M}$ with the obvious face maps. 
An element $h \in H_n$ fixes a vertex $\{\alpha\} \in \vert \mathcal{M} \vert$ if and only if $s \alpha h = s \alpha$ for all $s \in S$ if and only if $h$ fixes $S\alpha$, so the stabiliser $(H_n)_\alpha$ may only permute the finite set $S \setminus S\alpha$ and we may deduce
\begin{proposition}\label{prop:vmvHasFiniteIsotropy}
Stabilisers of simplicies in $\vmv$ are finite.
\end{proposition}
We now build up to the the proof that $\vmv$ is a model for $\uE H_n$ with a few lemmas.

\begin{proposition}\label{prop:vmvQfinite=>contractible}
If $Q \le H_n$ is a finite group then the fixed point set $\vmv^Q$ is non-empty and contractible.
\end{proposition}
\begin{proof} 
For all $q \in Q$, choose $\{z_0(q), \ldots, z_n(q)\}$ to be an n-tuple of natural numbers such that $(i,x)q = (i,x)$ whenever $i \ge z_x(q)$ for all $i$.  $Q$ then fixes all elements $(i,x) \in S$ with $i \ge \max_Q z_x(q)$.  Define a translation $t = t_1^{\max_Q z_1(q)} \cdots t_n^{\max_Q z_n(q)}$, $t \in \mathcal{M}^Q$ so $\{t \}$ is a vertex of $\vmv^Q$ and  $\vmv^Q \neq \emptyset$.  

If $\{m\}, \{n\} \in \vmv^Q$ then let $a, b \in T$ be two translations such that 
$$\phi(m) - \phi(n) = \phi(b) -  \phi(a)$$ 
(recall that for a translation $t$, $\phi(t)$ must be an n-tuple of positive numbers).  Thus $\phi(am) = \phi(bn)$, and since $am, bn \in \mathcal{M}$ there exist $n$-tuples $\{z_1, \ldots, z_n\}$ and $\{z_1^\prime, \ldots, z_n^\prime\}$ such that $am$ acts as a translation for all $(i,x) \in S$ with $i \ge z_x$ and $bn$ acts as a translation for all $(i,x) \in S$ with $i \ge z_x^\prime$.  Let 
$$c = t_1^{\max\{z_1, z_1^\prime\}} \ldots t_n^{\max\{z_n, z_n^\prime\}}$$
so that $cam = cbn$, further pre-composing $c$ with a large translation (for example that from the first section of this proof) we can assume that $cam = cbn \in \mathcal{M}^Q$, and $\{cam=cbn\} \in \, \vmv^Q$.  This shows the poset $\mathcal{M}^Q$ is directed and hence the simplicial realisation $\vert \mathcal{M}^Q \vert = \vmv^Q$ is contractible. 
\end{proof}
\begin{proposition}\label{prop:vmvQinfinite=>empty}
 If $Q \le H_n$ is an infinite group then $\vmv^Q = \emptyset$.
\end{proposition}
\begin{proof}
 Consider an infinite subgroup $Q \le H_n$ with $\vmv^Q \neq \emptyset$ and choose some vertex ${\{m\} \in \vmv^Q}$.  For any $q \in Q$, since $mq = m$ it must be that $\phi(m)+ \phi(q) = \phi(m)$ and so $\phi(q) = 0$, hence $Q$ is a subgroup of $\Sym_\infty \le H_n$.  Furthermore $Q$ must permute an infinite subset of $S$ (if it permuted just a finite set it would be a finite subgroup).  That $mq = m$ implies that this infinite subset is a subset of $S\setminus Sm$ but this is finite by construction.  So the fixed point subset $\vmv^Q$ for any infinite subgroup $Q$ is empty.
\end{proof}

\begin{corollary}\label{cor:Browns model is an EunderbarH}
 $\vmv$ is a model for $\uE H_n$.
\end{corollary}
\begin{proof}
Combine Propositions \ref{prop:vmvHasFiniteIsotropy}, \ref{prop:vmvQfinite=>contractible} and \ref{prop:vmvQinfinite=>empty}.
\end{proof}

\section{Finiteness conditions satisfied by \texorpdfstring{$H_n$}{Hn}} \label{section:finitenessConditionsSatisfied}

Recall from Section \ref{section:review of bredon finiteness} that a group $G$ is $\uFP_0$ if and only if it has finitely many conjugacy classes of finite subgroups.  $G$ satisfies the weaker quasi-$\uFP_0$ condition if and only if it has  finitely many conjugacy classes of subgroups isomorphic to a given finite subgroup.  
\begin{proposition}\label{prop:HisnotQuasiFP0}
 $H_n$ is not quasi-$\uFP_0$.
\end{proposition}
Before the above proposition is proved, we need a lemma.  In the infinite symmetric group $\Sym_\infty$ acting on the set $S$, elements can be represented by products of disjoint cycles.  We use the standard notation for a cycle: $(s_1, s_2, \ldots, s_m)$ represents the element of $\Sym_\infty$ sending $s_i \mapsto s_{i+1}$ for $i < n$ and $s_n \mapsto s_1$.  Any element of finite order in $H_n$ is contained in the infinite symmetric group $\Sym_\infty$ by the argument at the beginning of Section \ref{section:centralisersOfFiniteSubgroups}.  We say two elements of $\Sym_\infty$ have the same \emph{cycle type} if they have the same number of cycles of length $m$ for each $m \in \NN$.

\begin{lemma}\label{lemma:cycle type in H}
If $q$ is a finite order element of $H_n$ and $h$ is an arbitrary element of $H_n$, then $hqh^{-1}$ is the permutation given in the disjoint cycle notation by applying $h$ to each element in each disjoint cycle of $q$.  In particular, if $q$ is represented by the single cycle $(s_1, \ldots s_m)$, then $hqh^{-1}$ is represented by $(hs_1,\ldots,hs_m)$.

Furthermore, two finite order elements of $H_n$ are conjugate if and only if they have the same cycle type.
\end{lemma}
\begin{proof} The proof of the first part is analogous to \cite[Lemma 3.4]{Rotman-Groups}.  Let $q$ be an element of finite order and $h$ an arbitrary element of $H_n$.  If $q$ fixes $s \in S$ then $hqh^{-1}$ fixes $hs$.  If $q(i) = j$, $h(i) = k$ and $h(j) = l$, for $i,j,k,l \in S$, then $hqh^{-1}(k) = l$ exactly as required.

By the above, conjugate elements have the same cycle type.  For the converse, notice any two finite order elements with the same cycle type necessarily lie in $\Sym_r$ for some $r \in \NN$ so by \cite[Theorem 3.5]{Rotman-Groups} they are conjugated by an element of $\Sym_r$.
\end{proof}

\begin{proof}[Proof of Proposition \ref{prop:HisnotQuasiFP0}]
If $q$ is any order $2$ element of $H_n$, then $\{\id_{H_n}, q_1\}$ is a subgroup of $H_n$ isomorphic to $\ZZ_2$.  Choosing a collection of elements $q_i$ for each $i \in \NN_{\ge 1}$, so that $q_i$ has $i$ disjoint 2-cycles gives a collection of isomorphic subgroups which are all non-conjugate by Lemma \ref{lemma:cycle type in H}.
\end{proof}

\begin{proposition}\label{prop:cd=gd=n} $\ucd H_n = \ugd H_n = n$
\end{proposition}
\begin{proof}As described in the introduction, $H_n$ can be written as 
$$\Sym_\infty \longhookrightarrow H_n \longtwoheadrightarrow \ZZ^{n-1}$$
$\ugd \ZZ^{n-1} = n-1$ since a model for $\uE \ZZ^{n-1}$ is given by $\RR^{n-1}$ with the obvious action.  $\ugd \Sym_\infty = 1$ by \cite[Theorem 4.3]{LueckWeiermann-ClassifyingspaceForVCYC}, as it is the colimit of its finite subgroups each of which have geometric dimension $0$, and the directed category over which the colimit is taken has homotopy dimension $1$ \cite[Lemma 4.2]{LueckWeiermann-ClassifyingspaceForVCYC}.  $\ZZ^{n-1}$ is torsion free and so has a bound of $1$ on the orders of its finite subgroups and we deduce from \cite[Theorem 3.1]{Luck-TypeOfTheClassifyingSpace} that $\ugd H_n \le n-1 + 1 = n$.

To deduce the other bound, we use an argument due to Gandini \cite{Gandini-BoundingTheHomologicalFinitenessLength}.  Assume that $\ucd H_n \le n-1$.  By \cite[Theorem 2]{BradyLearyNucinkis-AlgAndGeoDimGroupsWithTorsion}
$$\cd_\QQ \le \ucd H_n = n-1 $$
In \cite[Theorem 5.1]{Brown-FinitenessPropertiesOfGroups}, it is proved that $H_n$ is $\FP_{n-1}$ (but not $\FP_n$), combining this with \cite[Proposition 1]{LearyNucinkis-BoundingOrdersOfFiniteSubgroups} we deduce that there is a bound on the orders of the finite subgroups of $H_n$, but this is obviously a contradiction.  Thus 
$$n \le \ucd H_n \le \ugd H_n \le n $$ 
\end{proof}

\begin{remark}
In \cite[Theorem 1]{DegrijsePetrosyan-GeometricDimensionForVCYC}, it is proved that for every countable elementary amenable group $G$ of finite Hirsch length $h$, $\uugd G \le h + 2$, (see the beginning of \cite{HillmanLinnell-ElemAmenofFinHirschLenAreLF-by-VS} for a precise definition of Hirsch length for elementary amenable groups).  From this we may deduce that since the Hirsch length of $H_n$ is $h(H_n) = n-1$,
$$\uugd H_n \le n+1$$ 

In \cite[Corollary 5.4]{LueckWeiermann-ClassifyingspaceForVCYC}, it is proved that $\uugd G \ge \ugd G - 1$ for any group $G$.  Thus we deduce
$$n-1 \le \uugd H_n \le n+1 $$
\end{remark}
We finish with the following question.

\medbreak\textbf{Question.}\quad What is the exact geometric dimension of Houghton's group $H_n$ with respect to the family of virtually cyclic subgroups?

\providecommand{\MR}{\relax\ifhmode\unskip\space\fi MR }

\end{document}